\documentclass{article}
\usepackage{amsthm,amssymb}
\newcommand{\HO}{\hbox{\rm Homth}}
\theoremstyle{plain}
\newtheorem{theorem}{Theorem}
\newtheorem{corollary}[theorem]{Corollary}
\newtheorem{lemma}[theorem]{Lemma}
\theoremstyle{definition}

\title{Abundance of Progressions in a Commutative\\ Semigroup by Elementary Means}

\date{}
\author{Sayan Goswami
\footnote{Department of Mathematics, 
           University of Kalyani, 
           Kalyani-741235,
           Nadia, West Bengal, India
           {\tt sayan92m@gmail.com}}
\and
Subhajit Jana
\footnote{Department of Mathematics, 
          University of Kalyani, 
          Kalyani-741235,
          Nadia, West Bengal, India
          {\tt suja12345@gmail.com}}
}

\begin{document}
\maketitle

\begin{abstract}
Furstenberg and Glasner proved that for an arbitrary
$k\in\mathbb{N}$, any piecewise syndetic set contains a $k$-term arithmetic
progression and such collection is also piecewise syndetic in $\mathbb{Z}$.
They used the algebraic structure of $\beta\mathbb{N}$. The above result
was extended for arbitrary semigroups by Bergelson and Hindman, again
using the structure of the Stone-\v{C}ech compactification of a general
semigroup. Beiglb\"ock provided an elementary proof of the above
result and asked whether the combinatorial argument in his proof can
be enhanced in a way which makes it applicable to a more abstract
setting.

In this work we extend that technique of Beiglb\"ock
in commutative semigroups.
\end{abstract}

\maketitle

\section{Introduction}

A subset $S$ of $\mathbb{Z}$ is called {\it syndetic\/} if
there exists $r\in\mathbb{N}$ such that $\bigcup_{i=1}^{r}(S-i)
=\mathbb{Z}$.
Again a subset $S$ of $\mathbb{Z}$ is called {\it thick\/}
if it contains arbitrarily long intervals. Sets which can be
expressed as the intersection of thick and syndetic sets are called {\it piecewise
syndetic\/}. 

For a general commutative semigroup $(S,+)$, a set $A\subseteq S$
is said to be syndetic in $(S,+)$, if there exist a finite nonempty set $F\subseteq S$
such that $\bigcup_{t\in F}(-t+A)=S$ where $-t+A=\{s\in S:t+s\in A\}$. A set $A\subseteq S$ is said
to be thick if for every finite nonempty set $E\subseteq S$, there exists an
element $x\in S$ such that $E+x\subseteq A$. A set $A\subseteq S$
is said to be piecewise syndetic set if there exist a finite nonempty set $F\subseteq S$
such that $\bigcup_{t\in F}(-t+A)$ is thick in $S$.
It can be proved that a piecewise syndetic set is the intersection
of a thick set and a syndetic set \cite[Theorem 4.49]{HS}.

One of the famous Ramsey theoretic results is van der Waerden's
Theorem \cite{V} which states that, given $r,k\in\mathbb{N}$ there is
some $l\in\mathbb{N}$ such that one cell of any partition
$\{C_{1},C_{2},\ldots,C_{r}\}$ of $\{1,2,\ldots,l\}$ contains an arithmetic
progressions of length $k$. It follows from van der Waerden's Theorem
that any piecewise syndetic subset $A$ of $\mathbb{N}$ contains arbitrarily
long arithmetic progressions. To see this, pick finite $F\subseteq\mathbb{N}$
such that $\bigcup_{t\in F}(-t+A)$ is thick in $\mathbb{N}$. Let $r=|F|$ and
let a length $k$ be given. Pick $l$ as guaranteed for $r$ and $k$ and pick
$x$ such that $\{1,2,\ldots,l\}+x\subseteq\bigcup_{t\in F}(-t+A)$. For
$t\in F$, let $C_t=\{y\in\{1,2,\ldots,l\}:y+x\in(-t+A)\}$. Pick
$a$ and $d$ in $\mathbb{N}$ and $t\in F$ such that
$\{a,a+d, a+2d,\ldots,a+(k-1)d\}\subseteq C_t$ and let 
$a'=a+x+t$.
Then $\{a',a'+d, a'+2d,\ldots,a'+(k-1)d\}\subseteq  A$.

A {\it homothetic\/} copy of a finite set $F$ in a commutative semigroup
$(S,+)$, is of the form $a+n\cdot F=\{a+n\cdot x:x\in F\}$, where $a\in S$, $n\in\mathbb{N}$
and $n\cdot x$ is the sum of $x$ with itself $n$ times.
Given a set $A\subseteq S$, we denote by $\HO_A(F)$ the set of 
homothetic copies of $F$ that are contained in $A$.

Furstenberg and E. Glasner \cite{FG} proved that if $S$ is a
piecewise syndetic subset of $\mathbb{Z}$ and $l\in\mathbb{N}$ then
the set of all length $l$ progressions contained in $S$ is also
large.
\begin{theorem}
\label{FGthm}Let $k\in\mathbb{N}$ and assume that $S\subseteq\mathbb{Z}$
is piecewise syndetic. Then $\{(a,d)\,:\,a,a+d,\ldots,a+kd\in S\}$
is piecewise syndetic in $\mathbb{Z}^{2}$.
\end{theorem}

Furstenberg and Glasner's proof was algebraic in nature.
Beiglb\"ock provided an elementary proof of Theorem \ref{FGthm} in \cite{B}.  
In \cite[Theorem 3.7]{BH},
Theorem \ref{FGthm} was extended to more general semigroups and for
various other notions of largeness using ultrafilter techniques. 

\begin{theorem}
\label{BHthm} Let $(S,\cdot)$ be a semigroup, let $l\in\mathbb{N}$,
let $E$ be a subsemigroup of $S^{l}$ with $\left\{ (a,a,\ldots,a):a\in S\right\} \subseteq E$,
and let $I$ be an ideal of $E$. 
If $M$ is a piecewise syndetic subset of $S$, then $M^{l}\cap I$
is piecewise syndetic in $I$.
\end{theorem}

\begin{proof}
\cite[Theorem 3.7(a)]{BH}.
\end{proof}

In this article we will extend that technique of Beiglb\"ock to commutative
semigroups and provide an elementary proof of a special case of Theorem
\ref{BHthm} which is Corollary \ref{Maincor}.

Conventionally $[t]$ denotes the set $\{1,2,\ldots,t\}$ and words
of length $N$ over the alphabet $[t]$ are the elements of $[t]^{N}$.
A {\it variable word\/} is a word over $[t]\cup\{*\}$ in which $*$ occurs
at least once and $*$ denotes the variable. A combinatorial line
is denoted by $L_{\tau}=\{\tau(1),\tau(2),\ldots,\tau(t)\}$ where
$\tau(*)$ is a variable word and for 
$i\in\{1,2,\ldots,t\}$, $\tau(i)$ is obtained by replacing
each occurrence of $*$ by $i$.

The following theorem is due to Hales and Jewett.
\begin{theorem}
\label{HJthm} For all values $t,r\in\mathbb{N}$, there exists a
number $HJ(r,t)$ such that, if $N\geq HJ(r,t)$ and $[t]^{N}$ is
$r$ colored then there will exists a monochromatic combinatorial
line.
\end{theorem}

\begin{proof} \cite[Theorem 1]{HJ}. \end{proof}

The above theorem may be strengthening by requiring that the variable
word include at least one constant.

\begin{lemma}
\label{HJnew} Let $r,m\in\mathbb{N}$. There exists 
some $N\in\mathbb{N}$ such that whenever the set of length $N$ words over
$[m]$ is r-colored, there is a variable word $w(*)$ such that
$w(*)$ begins and ends with a constant, and $\{w(a):a\in [m]\}$ is monochromatic.
\end{lemma}

\begin{proof} This is an immediate consequence of 
\cite[Lemma 14.8.1]{HS}. \end{proof}

The following lemma will be used in our proof of the main theorem.

\begin{lemma}
\label{lemHomo} Let $(S,+)$ be a commutative semigroup, 
let $F$ be a finite nonempty subset of $S$, and let $r\in\mathbb{N}$. 
There exists a finite subset $A$ of $S$ such that
for any $r$-coloring of $A$, there exist a monochromatic homothetic
copy of $F$.
\end{lemma}

\begin{proof} Let $m=|F|$ and enumerate $F$ as 
$\{s_1,s_2,\ldots,s_m\}$. Let $N$ be as guaranteed by
Lemma \ref{HJnew} for $r$ and $m$. Let
$A=\{\sum_{j=1}^Na_j:$ each $a_j\in F\}$ and let
$\chi:A\to\{1,2,\ldots,r\}$ be an $r$-coloring of $A$.
Define an $r$-coloring $\chi'$ of $[m]^N$ by
$\chi'(x)=\chi\big(\sum_{j=1}^Nx(j)\big)$ for 
$x=\big(x(1),x(2),\ldots,x(N)\big)\in [m]^N$.
Pick a variable word $w(*)=w_1w_2\cdots w_N$ which begins with a constant
and pick $l\in\{1,2,\ldots,r\}$ such that for each
$t\in\{1,2,\ldots,m\}$, $\chi'\big(w(t)\big)=l$.

Let $I=\{j\in\{1,2,\ldots,N\}:w_j=*\}$.
Let $a=\sum_{j\in[N]\setminus I}s_{w_j}$ and
let $n=|I|$.  Since $w(*)$ begins with a constant, $[N]\setminus I\neq
\emptyset$ so $a\in S$.  Then for 
$t\in \{1,2,\ldots,m\}$,
$l=\chi'\big(w(t)\big)=\chi(\sum_{i\in[N]\setminus I}s_{w_i}+
\sum_{i\in I}s_t)=\chi(a+n\cdot s_t)$. \end{proof}

For a commutative semigroup $(S,+)$, consider 
$(S\times\mathbb{N},+)$ with operation defined by $(s_{1},n_{1})+(s_{2},n_{2})=(s_{1}+s_{2},n_{1}+n_{2})$.
Then $(S\times\mathbb{N},+)$ is a commutative semigroup.

Using essentially the same proof that established that piecewise syndetic
subsets of $\mathbb{N}$ contain arbitrarily long arithmetic 
progressions, one can derive the following theorem from Lemma
\ref{lemHomo}. (We will not be using this theorem, which is
also a consequence of Theorem \ref{thmMain}.)

\begin{theorem}
Let $(S,+)$ be a commutative semigroup and let $F$
be a finite nonempty subset of  $S$.
Then for any piecewise syndetic set $M\subseteq S$
the collection $\{(a,n)\in S\times\mathbb{N}:\,a+n\cdot F\subseteq M\}$
is non empty.
\end{theorem}

The following is the main result of this paper.

\begin{theorem}
\label{thmMain} Let $(S,+)$ be a commutative semigroup, let $F$ be a
finite nonempty subset of $S$, and let $M$ be a piecewise syndetic subset
of $S$. Then the collection $\{(a,n)\in S\times\mathbb{N}:\,a+n\cdot F\subseteq M\}$
is piecewise syndetic in $(S\times\mathbb{N},+)$.
\end{theorem}

\section{Proof of Theorem \ref{thmMain}}

In this section we will prove Theorem \ref{thmMain}, enhancing the
combinatorial arguments used in \cite{B}. 
The following lemma
will be needed for our purposes. This lemma is
\cite[Lemma 4.6($I'$)]{BG}, and was proved by using the
algebraic structure of the Stone-\v{C}ech compactification of an arbitrary
semigroup. Since we have promised a combinatorial proof, we will
give a purely elementary proof of the above lemma for commutative
semigroups. We thank the anonymous referee for simplifying the
proof.

\begin{lemma}
\label{lemBG} Let $(S,\cdot)$ and $(T,.)$ be semigroups, let
$\varphi:S\to T$ be a homomorphism, and let $A\subseteq$S. 
If $A$ is piecewise syndetic in $S$ and $\varphi(S)$ is piecewise
syndetic in $T$, then $\varphi(A)$ is piecewise syndetic
in $T$.
\end{lemma}

\begin{proof} Assume that $A$ is piecewise syndetic in $S$
and $\varphi(S)$ is piecewise syndetic in $T$.
Pick finite $F\subseteq S$ such that $\bigcup_{s\in F}(-s+A)$
is thick in $S$ and pick finite $E\subseteq T$ such that
$\bigcup_{t\in E}\big(-t+\varphi(S)\big)$
is thick in $T$.  To show that $\varphi(A)$ is 
piecewise syndetic in $T$, it suffices to show that
$\bigcup_{j\in\varphi(F)+E}\big(-j+\varphi(A)\big)$
is thick in $T$.  

To see this, let $K$ be a finite nonempty subset of $T$.
Pick $n\in T$ such that $K+n\subseteq \bigcup_{t\in E}\big(-t+\varphi(S)\big)$.
For each $k\in K$, pick $t_k\in E$ such that $t_k+k+n\in\varphi(S)$ and 
pick $b_k\in S$ such that $\varphi(b_k)=t_k+k+n$. Let 
$B=\{b_k:k\in K\}$ and pick $m\in S$ such that $B+m\subseteq
\bigcup_{s\in F}(-s+A)$.  We claim that
$$K+n+\varphi(m)\subseteq \bigcup_{j\in\varphi(F)+E}\big(-j+\varphi(A)\big)$$
so let $k\in K$ be given. Pick $s\in F$ such that
$s+b_k+m\in A$. Then $\varphi(s)+t_k\in\varphi(F)+E$ 
and $(\varphi(s)+t_k)+k+n+\varphi(m)=
\varphi(s+b_k+m)\in\varphi(A)$. \end{proof}

\begin{lemma}
\label{lemMPS} If $M\subseteq S\times\mathbb{N}$ is piecewise syndetic
in $S\times\mathbb{N}$, 
then for any $x,y\in S$ and $t\in\mathbb{N}$,
$\{(a+nx+y,n\cdot t):(a,n)\in M\}$
is piecewise syndetic in $S\times\mathbb{N}$.
\end{lemma}

\begin{proof}
For a fixed $x\in S$, let $\varphi_{x}:S\times\mathbb{N}\rightarrow S\times\mathbb{N}$
be a homomorphism defined by $\varphi_{x}(a,n)=(a+nx,n)$. We claim
that $\varphi_{x}(S\times\mathbb{N})$ is thick in $S\times\mathbb{N}$
and so piecewise syndetic and hence this transformation preserves
piecewise syndeticity.

To prove the claim, let us take a finite subset $F=\{(s_{1},n_{1}),(s_{2},n_{2}),\ldots,\break
(s_{m},n_{m})\}$
of $S\times\mathbb{N}$  and observe that  $F+\big((n_{1}+n_{2}+\ldots+n_{m}+1)x,1\big)\subseteq
\varphi_{x}(S\times\mathbb{N})$.

Now for a fixed $t\in\mathbb{N},$ let $\chi_{t}:S\times\mathbb{N}\rightarrow S\times\mathbb{N}$
be a homomorphism defined by $\chi_{t}(a,n)=(a,n\cdot t)$. We claim
that, $\chi_{t}(S\times\mathbb{N})$ is syndetic in $S\times\mathbb{N}$,
in particular piecewise syndetic. Therefore this transformation also
preserves piecewise syndeticity.

To prove the claim, pick $x\in S$. We show that $S\times\mathbb{N}\subseteq\bigcup_{j=1}^t
(-(x,j)+\chi_{t}(S\times\mathbb{N}))$. So let $(a,n)\in S\times\mathbb{N}$ be given.
Pick $j\in\{1,2,\ldots,t\}$ such that $t$ divides $n+j$ and let
$m=\frac{n+t}{m}$. Then $(x,j)+(a,n)=\chi_t(x+a,m)$.

Again for any fixed $y\in S$, we claim that $\psi_{y}:S\times\mathbb{N}\to S\times\mathbb{N}$
defined by $\psi_{y}(a,n)=(a+y,n)$ preserves piecewise syndeticity.
To see this, let $A\subseteq S\times\mathbb{N}$ be piecewise syndetic in $S\times\mathbb{N}$.
Pick a finite nonempty subset $E$ of $S\times\mathbb{N}$ such that
$\bigcup_{(a,n)\in E}(-(a,n)+A)$ is thick.  Since
$$\textstyle \bigcup_{(a,n)\in E}(-(a,n)+A)\subseteq \bigcup_{(a,n)\in E}\big(-(a+y,n)+\psi_y(A)\big)\,,$$
we have that $\bigcup_{(a,n)\in E}\big(-(a+y,n)+\psi_y(A)\big)$ is thick, so
$\psi_y(A)$ is piecewise syndetic.

As $(a+nx+y,n.t)=\psi_{y}\circ\chi_{t}\circ\varphi_{x}(a,n)$ and
all the three maps preserves piecewise syndeticity, we have the desired
result.
\end{proof}

Now we will prove the main theorem.

\begin{proof}[Proof of Theorem \ref{thmMain}]
 As $M$ is piecewise syndetic, pick a finite nonempty set $E\subseteq S$
such that $\bigcup_{t\in E}(-t+M)$ is thick and let $r=|E|$.
Pick by Lemma \ref{lemHomo} a finite subset $A$ of $S$ such 
that, whenever $A$ is $r$-colored, there exists a monochromatic
homothetic copy of $F$.  Then, given $a\in S$ and $n\in\mathbb{N}$,
if $a+n\cdot A$ is $r$-colored, there exist $b\in S$ and $t\in\mathbb{N}$
such that $a+n\cdot(b+t\cdot F)$ is monochromatic.

Let $B=\{(a,n)\in S\times\mathbb{N}:a+n\cdot A\subseteq\bigcup_{t\in E}(-t+M)\}$.
We claim that $B$ is thick in $S\times\mathbb{N}$. To see this, let
$C$ be a finite nonempty subset of $S\times \mathbb{N}$.
Let $G=\bigcup_{(a,n)\in C}(a+(n+1)\cdot A)$.  By the thickness
of $\bigcup_{t\in E}(-t+M)$, pick $z\in S$ such that
$G+z\subseteq\bigcup_{t\in E}(-t+M)$. Then, given $(a,n)\in C$,
$a+(n+1)\cdot A\subseteq G$ so
$a+z+(n+1)\cdot A\subseteq \bigcup_{t\in E}(-t+M)$
and thus $(a,n)+(z,1)\in B$.

Let $s=|\HO_A(F)|$, enumerate $\HO_A(F)$ as $\{M_1,M_2,\ldots,M_s\}$,
and enumerate $E$ as $\{t_1,t_2,\ldots,t_r\}$.
Note that, for any $(a,n)\in B$, there are some $i\in\{1,2,\ldots,r\}$
and some $(b,u)\in S\times\mathbb{N}$ such that 
$$a+n\cdot(b+u\cdot F)\subseteq a+n\cdot A\cap(-t_i+M)\,.$$

Define $\varphi:B\to E\times \HO_A(F)$ as follows. First pick
the least $i\in\{1,2,\ldots,r\}$ such that there exist $b\in S$ and
$u\in \mathbb{N}$ such that $a+n\cdot (b+u\cdot F)\subseteq (a+n\cdot A)\cap(-t_i+M)$ and
note that $b+u\cdot F\in \HO_A(F)$. Now pick the least $j\in\{1,2,\ldots,s\}$ 
such that $a+n\cdot M_j\subseteq (a+n\cdot A)\cap(-t_i+M)$. 
Then define $\varphi(a,n)=(t_i,M_j)$.

Now as $E\times \HO_{A}(F)$ is finite, the mapping $\varphi$ gives
a finite coloring of $B$ so pick $(i,j)\in \{1,2,\ldots,r\}\times
\{1,2,\ldots,s\}$ such that $Q=\{(a,n)\in B:\varphi(a,n)=(t_i,M_j)\}$
is piecewise syndetic. (We are using here the elementary fact that
if the union of finitely many sets is piecewise syndetic, then one
of them is.)

Choose some $b\in S$ and $u\in\mathbb{N}$ such that $b+u\cdot F=M_{j}$.
then $(a,n)\in Q$ implies $a+n\cdot(b+u\cdot F)\subseteq(-t_{i}+M)$
and hence $a+t_{i}+n\cdot(b+u\cdot F)\subseteq M$.
Now by Lemma \ref{lemMPS}, $\widetilde{Q}=\{(a+n\cdot b+t_{i},n\cdot u)\,:\,(a,n)\in Q\}$
is piecewise syndetic. For $(a_{1},n_{1})\in\widetilde{Q}$ , $a_{1}+n_{1}\cdot F\subseteq M$
and this proves the theorem.
\end{proof}

The following corollary gives an elementary proof of a special case
of Theorem \ref{BHthm}.

Here we denote the set of all homothetic copies of $F=\{s_{1},s_{2},\ldots,s_{l}\}$
by $HC_{F}$, i.e. the set $\{(a+n\cdot s_{1},a+n\cdot s_{2},\ldots,a+n\cdot s_{l})\,:\,a\in S$ and $n\in\mathbb{N}\}$.
It is easy to check that $HC_{F}$ is a subsemigroup of $S^{l}$.

\begin{corollary}
\label{Maincor} Let $(S,+)$ be a commutative semigroup and
$F\subseteq S$ be a given finite nonempty set of cardinality $l$. Then for
any piecewise syndetic subset $M\subseteq S$, $M^{l}\cap HC_{F}$
is piecewise syndetic in $HC_{F}$.
\end{corollary}

\begin{proof}
Let us define a surjective homomorphism $\varphi:S\times\mathbb{N}\to HC_{F}$
by, $\varphi(a,n)=(a+n\cdot s_{1},a+n\cdot s_{2},\ldots,a+n\cdot s_{l})$.

Now let $C=\{(a,n)\in S\times\mathbb{N}:a+n\cdot F\subseteq M\}$. Then $C$
is piecewise syndetic in $S\times\mathbb{N}$ by Theorem \ref{thmMain}.
Now $\varphi(C)\subseteq M^{l}\cap HC_{F}$ and as $C$ is
piecewise syndetic in $S\times\mathbb{N}$, $\varphi(C)$ is piecewise
syndetic in $HC_{F}$ by Lemma \ref{lemBG} and so $M^{l}\cap HC_{F}$ is piecewise syndetic
in $HC_{F}$.
This proves the claim.
\end{proof}

\textbf{Acknowledgements:} The first author acknowledges the grant UGC-NET
SRF fellowship with id no. 421333 of CSIR-UGC NET December 2016. The
second author acknowledges the grant CSIR-SRF fellowship with file
no.\hfill\break 09/106(0149)/2014-EMR-1. We also acknowledge the anonymous referee
for several helpful comments on the paper.

\bibliographystyle{alpha}

\end{document}